\documentclass[reqno]{amsart}
\usepackage[numbers,square,comma,sort&compress]{natbib}
\usepackage{amsmath,amssymb}
\usepackage{color}

\newcommand{\real}{\mathbb{R}}
\newcommand{\nat}{\mathbb{N}}

\newcommand{\oi}{(0,\infty)}

\newcommand{\rn}{\real^N}
\newcommand{\intrn}{\int_{\real^N}}

\newcommand{\ep}{\epsilon}
\newcommand{\de}{\delta}

\newcommand{\ffy}{\varphi_\infty}
\newcommand{\D}{\Delta}
\newcommand{\lam}{\lambda}
\newcommand{\ly}{\lambda_\infty}
\newcommand{\ls}{\lambda_*}

\newcommand{\e}{\mathrm{e}}

\newcommand{\diff}{\,\mathrm{d}}

\newcommand{\Ve}{\Vert}

\renewcommand\emptyset{\mbox{\Large \o}}
\newcommand{\sm}{\setminus}
\newcommand{\x}{\times}
\newcommand{\les}{\leqslant}
\newcommand{\ges}{\geqslant}

\newcommand{\calZ}{\mathcal{Z}}
\newcommand{\calS}{\mathcal{S}}
\newcommand{\calC}{\mathcal{C}}
\newcommand{\calZn}{\calZ_n}
\newcommand{\calCn}{\calC_n}
\newcommand{\calCo}{\calC_0}

\DeclareMathOperator \rge{rge}
\DeclareMathOperator \proj{P}

\newtheorem{thm}{Theorem}
\newtheorem{lem}[thm]{Lemma}
\newtheorem{prop}[thm]{Proposition}

\newtheorem{rem}{Remark}
\newtheorem{ex}{Example}

\begin{document}

\title[Asymptotically linear Schr\"odinger equations]
{Global bifurcation for asymptotically linear Schr\"odinger equations}
\author{Fran\c cois Genoud}

\thanks{This work was supported by the Engineering and Physical Sciences Research Council, [EP/H030514/1].}


\address{Department of Mathematics and the Maxwell Institute
for Mathematical Sciences \\ Heriot-Watt University \\
Edinburgh \\ EH14 4AS \\ Scotland.}
\email{F.Genoud@hw.ac.uk}

\begin{abstract}
We prove global asymptotic bifurcation for a very general class of 
asymptotically linear Schr\"odinger equations 
\begin{equation}\label{1}
\left\{ \begin{array}{lr}
\D u +  f(x,u)u = \lam u \quad \text{in} \ \rn,\\
u \in H^1(\rn)\sm\{0\}, \quad N \ges 1.
\end{array}\right.
\end{equation}
The method is topological, based on recent
developments of degree theory. We use the inversion $u\to v:= u/\Vert u\Vert_X^2$ in 
an appropriate Sobolev space $X=W^{2,p}(\rn)$, and
we first obtain bifurcation from the line of trivial solutions 
for an auxiliary problem in the variables $(\lambda,v) \in {\mathbb R} \x X$. 
This problem has a lack of compactness and of regularity, requiring
a truncation procedure. 
Going back to the original problem, we obtain global branches of 
positive/negative solutions `bifurcating from infinity'. 
We believe that, for the values of $\lambda$ covered
by our bifurcation approach, the existence result we obtain for
positive solutions of \eqref{1} is the most general so far
\end{abstract}

\setcounter{equation}{0}

\maketitle

\section{Introduction}

We consider nonlinear Schr\"odinger equations of the form
\begin{equation}\label{nls}
\left\{ \begin{array}{lr}
\D u +  f(x,u)u = \lam u \quad \text{in} \ \rn,\\
u \in H^1(\rn)\sm\{0\}, \quad N \ges 1,
\end{array}\right.
\end{equation} 
where the nonlinearity is supposed to be {\em asymptotically linear} in the sense that
there exists $f_\infty \in C(\rn)$ such that $f(x,s) \to f_\infty(x)$ as $|s| \to \infty$ 
for all $x \in \rn$. 
A solution to \eqref{nls} is a couple $(\lambda,u) \in \real \x H^1(\rn) \sm \{0\}$
satisfying the elliptic equation in the sense of distributions. 
We will establish global bifurcation for \eqref{nls} in $\real \x W^{2,p}(\rn)$ 
with $p \in [2,\infty) \x (\tfrac12 N,\infty)$, 
yielding solutions with $u\in H^1(\rn)$. More precisely, we will show that there exists 
unbounded connected sets $\calS^\pm\subset \real \x W^{2,p}(\rn)$ 
of positive/negative solutions of \eqref{nls}.
We use a method introduced in \cite{sz}, based on a topological degree for compact
perturbations of $C^1$ Fredholm maps. We have applied this method recently to 
problem \eqref{nls} in dimension $N=1$ -- actually, on the half-line -- \cite{g2}, 
by establishing global bifurcation in $H^2\oi$. The present work extends the results
of \cite{g2} in various directions, as explained in more detail below.

\medskip

To present the method and discuss our results, let us first rewrite \eqref{nls} as
\begin{equation}\label{rewrite}
\D u + f_\infty(x)u + g(x,u)u = \lam u \quad \text{with} \quad g(x,s) := f(x,s)-f_\infty(x).
\end{equation}
Following Rabinowitz \cite{r} and Toland \cite{t1}, 
we use the inversion $u\to v:= u/\Vert u\Vert^2$
(with $\Vert\cdot\Vert:=\Vert\cdot\Vert_{W^{2,p}(\rn)}$) to get
\begin{equation}\label{invert}
\D v + f_\infty(x)v + g(x,v/\Vert v\Vert^2)v = \lam v,
\end{equation}
for which we shall prove global bifurcation from the point $(\ly,0)$ in $\real \x W^{2,p}(\rn)$.
Here, $\ly>0$ is characterized as the principal eigenvalue of the linear problem
\begin{equation}\label{lin}
\D u + f_\infty(x)u = \lam u,
\end{equation}
known as the `asymptotic linearization' of \eqref{nls}.
Returning to the original variable will yield an unbounded connected set 
of solutions $(\lambda,u)$ of \eqref{nls}, satisfying
$$
\Vert u\Vert\to\infty \quad\text{as}\quad \lambda\to\lambda_\infty.
$$
This behaviour is often referred to as {\em asymptotic bifurcation} and was established
in \cite{t1,r} in the context of boundary value problems in bounded domains. 
The core of the method there is Rabinowitz's global bifurcation theory. The present
situation is more difficult because the domain is unbounded and problem \eqref{nls}
is not asymptotically linear in the rigorous sense -- which amounts to saying that the auxiliary problem \eqref{invert} is not Fr\'echet differentiable at $v=0$. 
Following Stuart and Zhou \cite{sz}, we therefore resort to a truncation procedure to transform 
\eqref{invert} into a problem for which global bifurcation can be obtained by applying an
abstract theorem proved in \cite{sz}, based on recent developments of degree theory \cite{rs}.

\medskip

Our precise hypotheses on the nonlinearity are the following.

\begin{itemize}
\item[(f1)] $f\in C(\rn\x\real) \cap L^\infty(\rn\x\real)$ and $f(x,\cdot)\in C^1(\real)$ for all $x\in\rn$.
\item[(f2)] We have $f(x,s)\to f_0(x) := f(x,0)$ as $s\to0$, uniformly for $x\in\rn$.
\item[(f3)] There exists a function $f_\infty\in C(\rn)$ such that 
$f(x,s)\to f_\infty(x)$ as $|s|\to\infty$, uniformly for $x\in\rn$.
Furthermore, there exists $\de>0$ such that
$$0 < f_0(x) + \de < \ls := \limsup_{|x|\to\infty}f_\infty(x) < \infty \quad\text{for all} \ x\in\rn.$$
\item[(f4)] $0\les f(x,s)\les f_\infty(x)$ for all $(x,s)\in\rn\x\real$.
\item[(f5)] $\partial_2 f(\cdot,s)\in L^\infty(\rn)$ for all $s\in\real$ and 
$\{\partial_2 f(x,\cdot)\}_{x\in\rn}$ is equicontinuous.
\item[(f6)] We have
$$
\lambda_\infty := 
- \inf_{u\in H^1(\rn)\sm\{0\}}
\frac{\intrn \{ |\nabla u|^2-f_\infty(x)u^2 \} \diff x}{\intrn u^2\diff x} >\ls.
$$
\end{itemize}

\begin{rem}\label{lim}\rm
\item[(i)] Note that the hypothesis (f3) prevents $f$ from being independent of $s$.
\item[(ii)] Assumption (f6) is satisfied provided $f_\infty \ \substack{\ges \\ \not\equiv} \ \ls$ 
on $\rn$, in which case the `$\limsup$' in (f3) is actually a `$\lim$': $\lim_{|x|\to\infty} f_\infty(x)=\ls$.
\end{rem}

\begin{thm}\label{thm} Let $f$ satisfy the hypotheses $(\mathrm{f}1)$ to $(\mathrm{f}6)$
and $p \in [2,\infty) \x (\tfrac12 N,\infty)$. There exist two connected sets 
$\calS^\pm \subset \real \x W^{2,p}(\rn)$ of positive/negative solutions of \eqref{nls}
with the following properties.
\item[$(\mathrm{i})$] $\proj\calS^\pm=(\ls,\ly)$, where
$\proj(\lam,u):=\lam$ for all $(\lam,u)\in \real \x W^{2,p}(\rn)$. 
\item[$(\mathrm{ii})$] $\calS^\pm$ are bounded away from $\real\x\{0\}$ in 
$\real\x W^{2,p}(\rn)$.
\item[$(\mathrm{iii})$] If $\{(\lam_n,u_n)\}\subset\calS^\pm$ is such that 
$\lam_n\to\lam$ as $n\to\infty$, then
$$
\lim_{n\to\infty} \Ve u_n \Ve_{L^p(\rn)}=
\lim_{n\to\infty} \Ve u_n \Ve_{L^\infty(\rn)}=\infty \ \Longleftrightarrow \ \lam=\ly.
$$
\item[$(\mathrm{iv})$] For all $(\lam,u)\in\calS^\pm$, we have
$u \in C^1(\rn)$ and $|\nabla u(x)|, u(x)\to0$ as $|x|\to\infty$.
\end{thm}

\begin{ex}\rm
The function 
$$f(x,s) := \frac{1}{1+s}f_0(x) + \frac{s}{1+s}f_\infty(x)$$ satisfies (f1) to (f6) 
if $f_0,f_\infty \in C(\rn) \cap L^\infty(\rn)$ with $f_0,f_\infty\ges0, \ f_0\les f_\infty$, 
$f_\infty \ \substack{\ges \\ \not\equiv} \ \ls := \lim_{|x|\to\infty}f_\infty(x)$
and $f_0 + \de < \ls$ on $\rn$, for some $\de>0$.
\end{ex}

In our previous work on the one-dimensional case \cite{g2}, we had in mind applications to 
nonlinear waveguides and were therefore working under more restrictive assumptions.
In particular, we were dealing with a nonlinearity of the form $f(x,u^2)u$ with an
$f\in C(\real_+^2)$ satisfying similar assumptions to (f1)-(f6), but with $f_0 \equiv$ const.
As will be seen in the proof of Theorem~\ref{thm}, the nonlinearity $f(x,u)u$ in \eqref{nls}
can be handled similarly, without any symmetry assumption.
The restriction on $f_0$ was made in connection with our previous work on related problems
\cite{g1}, with a view to continuing the branch of solutions down to the line of trivial solutions,
where we expect it to meet the $\lam$\,-\,axis at the boundary of the essential spectrum of 
the linearization at $u=0$. 
The linearization at $u=0$ and the asymptotic linearization \eqref{lin} can respectively be written as
\begin{equation}\label{linear}
L_{0/\infty} u = \lam u,
\end{equation}
where $L_{0/\infty}: D(L_{0/\infty}):=H^2(\rn)\subset L^2(\rn) \to L^2(\rn)$ are
the Schr\"odinger operators\footnote{As shown in Proposition~\ref{eigen.prop} below,
we do not lose any generality by restricting the domain of these operators to $H^2(\rn)$.}
defined by
$$
L_{0/\infty} u := \D u + f_{0/\infty}(x)u.
$$
We refer the reader to \cite{trieste} for the spectral theory of Schr\"odinger operators we shall need here. 
In particular, under the above hypotheses, we have
$$
\ly=\sup\sigma(L_\infty), \ \sigma_\mathrm{e}(L_\infty) \subset (-\infty,\ls] 
\quad\text{and}\quad \sigma(L_0) \subset (-\infty,\ls),
$$
where $\sigma$ and $\sigma_\mathrm{e}$ denote 
the spectrum and the essential spectrum, respectively. 
(Note that $\sigma_\mathrm{e}(L_\infty)=(-\infty,\ls]$ if $f_\infty$ satisfies the assumption
in Remark~\ref{lim}(ii).)
In the present work, as in \cite{g2}, we make the hypothesis (f6), ensuring that $\ly$ is
an eigenvalue of \eqref{lin}, to obtain a branch of solutions `bifurcating from infinity'. 
The discussion regarding bifurcation from the line of trivial solutions relies on $L_0$, and
two different scenarios can occur, depending on whether 
$\sup\sigma_\mathrm{e}(L_0)<\sup\sigma(L_0)$
or $\sup\sigma_\mathrm{e}(L_0)=\sup\sigma(L_0)$. In the first case,
$\sup\sigma(L_0)$ is a simple eigenvalue, and the problem can be dealt with 
by standard bifurcation theory. The one-dimensional problem was
studied in great detail in \cite{js}. In the situation where 
$\sup\sigma_\mathrm{e}(L_0)=\sup\sigma(L_0)$ 
(which occurs for instance if $f_0 \equiv$ const.), 
we expect to observe bifurcation from the essential spectrum of $L_0$. 
This case deserves further attention and \cite{g2} can be seen as 
a first step in this direction.
However, the method used in \cite{g2} is very powerful and allows
one to deal with the present, much more general situation. 

Asymptotically linear Schr\"odinger equations have been extensively studied 
for the past 15 years, see e.g. \cite{sz96,z,sz99,j,jt,lz,ct,v,vw,lw,sz05,sz,zz,g2}. 
Most of these papers make use of variational methods to prove existence of solutions 
of various asymptotically linear problems, many of which can be put in the form \eqref{nls}.
Important developments were achieved in order to apply 
the mountain pass theorem without the so-called `superquadraticity condition', 
which is not satisfied by asymptotically linear nonlinearities. 
Two landmark contributions in this direction were \cite{sz99,j}, 
but many others followed.
In this context, fairly strong assumptions are usually made, 
requiring for instance radial symmetry, i.e. $f(x,s)=f(|x|,s)$, or simply
that $f$ does not depend on $x$, or that at least one of the 
linear problems \eqref{linear} is autonomous,
i.e. $f_0$ or $f_\infty$ is constant.
To ensure the compactness of minimizing sequences, 
the concentration-compactness principle is often used, 
requiring the existence of a limit problem as $|x|\to\infty$. 
At least, limits of $f_{0/\infty}$ as $|x|\to\infty$ are usually supposed to exist. 
In this respect, our work is probably most closely related to \cite{zz}, 
where both linearizations are non-autonomous (see also \cite{sz05} for a one-dimensional problem).
Typical results provide existence of positive
solutions obtained as critical points of an appropriate functional, for any given 
$\lam \in (\sup\sigma(L_0),\sup\sigma(L_\infty))$. 
As far as we know, the only bifurcation-type 
results were obtained in \cite{sz} for a problem with a potential well, where the parameter
$\lam$ appears with a weight.
Our method, derived from \cite{sz},
does not allow $\lam$ to go below $\ls$ since the Fredholm property needs
to be preserved along the branch (see Lemma~\ref{fredhops}), and so we do not cover
the full interval $(\sup\sigma(L_0),\sup\sigma(L_\infty))$.
However, for $\lam\in(\ls,\ly)$, we believe that our results are the most general so far. 
In particular, we do not need any symmetry or monotonicity assumptions, and
$f_{0/\infty}$ need not have limits as $|x|\to\infty$. 

\medskip

\noindent{\bf Notation and terminology.} 
Throughout the paper, the usual norm on $L^r(\rn)$ for $1\les r\les\infty$ 
will be denoted by $|\cdot|_r$. 
Completely continuous nonlinear mappings between Banach spaces will be
termed `compact'.


\section{Preliminary results}\label{prel}

We first establish some basic properties of solutions of \eqref{nls}.

\begin{prop}\label{reg} Suppose that $(\mathrm{f}1)$ to $(\mathrm{f}6)$ are satisfied.
\item[$(\mathrm{i})$] If $(\lam,u) \in \real \x H^1(\rn)$ satisfies \eqref{nls} 
then $u \in W^{2,p}(\rn)$ for all $p \in [2,\infty)$. 
In particular, $u \in C^1(\rn)$ with $|\nabla u(x)|, u(x)\to0$ as $|x|\to\infty$.
\item[$(\mathrm{ii})$] If $(\lam,u) \in \real \x W^{2,p}(\rn)$ satisfies \eqref{nls} 
with $\lam>\ls$ and $p \in [2,\infty) \x (\tfrac12 N,\infty)$, then $u \in H^1(\rn)$.
\end{prop}

\begin{proof} If $(\lam,u)$ is as in (i), $u \in W^{2,p}(\rn)$ for all $p \in [2,\infty)$ 
follows by standard elliptic regularity theory (see e.g. \cite{trieste}), 
using the fact that $f$ is bounded.

If $(\lam,u)$ is as in (ii), it follows by Lemma~\ref{expdecay} below that $u \in L^2(\rn)$ 
and so $u \in H^2(\rn)\subset H^1(\rn)$ by elliptic regularity.
\end{proof}

We now give some properties of the linear eigenvalue problem
\begin{equation}\label{eigen}
\begin{cases}
\D u + f_\infty(x)u = \lam u \quad\text{in} \ \rn,
\\
u \in H^1(\rn)\sm\{0\}, \quad \lam>0.
\end{cases}
\end{equation}
A number $\lam>0$ is called an {\em eigenvalue} if there exists a function $u \in H^1(\rn)\sm\{0\}$,
called an {\em eigenfunction}, such that $(\lam,u)$ satisfies \eqref{eigen} (in the sense
of distributions).

\begin{prop}\label{eigen.prop} Suppose that $(\mathrm{f}3)$ and 
$(\mathrm{f}6)$ are satisfied.
\item[$(\mathrm{i})$] If $(\lam,u) \in \real \x H^1(\rn)$ satisfies \eqref{eigen} 
then $u \in W^{2,p}(\rn)$ for all $p \in [2,\infty)$. 
In particular, $u \in C^1(\rn)$ with $|\nabla u(x)|, u(x)\to0$ as $|x|\to\infty$.
\item[$(\mathrm{ii})$] The number $\ly$ defined in $(\mathrm{f}6)$ is a simple eigenvalue of 
\eqref{eigen} and there exists a corresponding eigenfunction $\ffy \in H^1(\rn)$ with
$|\ffy|_2=1$ and $\ffy>0$ on $\rn$.
\end{prop}

\begin{proof} (i) follows by elliptic regularity. As discussed in the introduction,
(ii) follows from the spectral theory of Schr\"odinger operators (see e.g. \cite{trieste}).
\end{proof}


\section{Truncation and bifurcation}\label{bif}

For $p \in [2,\infty) \x (\tfrac12 N,\infty)$ fixed, we define the following Banach spaces:
\begin{alignat*}{10}
X&:=W^{2,p}(\rn) &\quad\text{with}\quad \Vert\cdot\Vert_X&:=\Vert\cdot\Vert_{W^{2,p}},\\
Y&:=L^p(\rn) &\quad\text{with}\quad \Vert\cdot\Vert_Y&:=|\cdot|_p.
\end{alignat*}
We start by giving an operator formulation of the inverted problem \eqref{invert}. 
For $\lambda\in\real$, we define $L(\lambda), G: X\to Y$ by
$$
L(\lambda)v := \D v + f_\infty(x)v - \lambda v \quad\text{and}\quad 
G(v) := g(x,v/\Vert v\Vert_X^2) v \quad  \text{for} \ v\neq0, \quad G(0) := 0.
$$
Here, as usual, $u\to g(x,u)$ denotes the Nemytskii operator generated by $g$. 
Under the hypotheses of Theorem~\ref{thm}, $G$ is well-defined and continuous.
Then \eqref{invert} is equivalent to 
\begin{equation}\label{opequ}
L(\lambda)v+G(v)=0.
\end{equation}
The following global bifurcation theorem is due to Stuart and Zhou.
For arbitrary Banach spaces $X$ and $Y$,
we denote by $B(X,Y)$ the space of bounded linear operators from $X$ to $Y$, and we let
$$
\Phi_0(X,Y) := \{L\in B(X,Y): L \ \text{is a Fredholm operator of index zero}\}.
$$

\begin{thm}\label{globif} {\rm \cite[Theorem~A.1]{sz}} 
Let $L\in C^1(J,\Phi_0(X,Y))$ where $J\subset\real$ is an open interval and $\lam_0\in J$ 
be such that 
$$
\dim\ker L(\lam_0) \ \text{is odd} \quad\text{and}\quad
[L'(\lam_0)\ker L(\lam_0)]\oplus \rge L(\lam_0)=Y. 
$$
Suppose that $K\in C(X,Y)$ is compact and Fr\'echet differentiable at $u=0$ with 
$K'(0)=0$.
Let 
$$\calZ := \{(\lam,u)\in J\x X: u\neq0 \ \text{and} \ L(\lam)u+K(u)=0\},$$ 
endowed with the metric inherited from $\real\x X$, and let $\calC$ denote 
the connected component of $\calZ\cup\{(\lam_0,0)\}$ such that $(\lam_0,0)\in\calC$. 
Then $\calC$ has at least one of the following properties$\,:$
\item[(a)] $\calC$ is an unbounded subset of $\real\x X;$
\item[(b)] $\overline{\calC}\cap[J\x\{0\}]\neq\{(\lam_0,0)\}$, where 
$\overline{\calC}$ denotes the closure of $\calC$ in $\real\x X;$
\item[(c)] either $\inf\proj\calC=\inf J$ or $\sup\proj\calC=\sup J.$
\end{thm}
In the present context, we would like to obtain informations about connected sets of solutions of
problem \eqref{opequ} by applying Theorem~\ref{globif} with $\lam_0=\ly$, $K=G$,
$X, \ Y$ and $L(\lam)$ as defined above.
Unfortunately, this cannot be done directly because
$G$ is not differentiable at $v=0$ (this can be proved similarly to \cite[Lemma~B.1]{sz}) 
and is not compact. 
Therefore, we introduce the following sequence of approximate problems. We define
$G_n:X\to Y$ by 
\begin{equation}\label{Gn}
G_n(v)(x) := \chi_n(x)G(v)(x) \quad \text{for} \ v\in X \ \text{and} \ n\in\nat := \{1,2,\dots\},
\end{equation}
where
$$
\chi_n(x) := 
\begin{cases}
1  & \text{if } \  |x|\les n,
\\
0  & \text{if } \ |x|>n.
\end{cases}
$$
Since the mapping $\lam\to L(\lam)$ is of class $C^1(\real,B(X,Y))$, 
the following lemma shows that $L(\lam)$ satisfies the hypotheses of Theorem~\ref{globif},
with $J=(\ls,\infty)$ and $\lam_0=\lam_\infty$.

\begin{lem}\label{fredhops}
Suppose that $(\mathrm{f}3)$ and $(\mathrm{f}6)$ are satisfied and let $J=(\ls,\infty)$.
\item[$(\mathrm{i})$] $L(\lambda)\in \Phi_0(X,Y)$ for all $\lambda\in J$.
\item[$(\mathrm{ii})$] For $\lambda=\ly\in J$, we have
$$\dim\ker L(\ly)=1 \quad\text{and}\quad [L'(\ly)\ker L(\ly)]\oplus \rge L(\ly)=Y.$$
\end{lem}

\begin{proof} These are standard properties of the Schr\"odinger operators 
$L(\lambda)$.
Part (i) follows for instance from \cite[Section~4]{jls}.
For any $\lambda\in\real$, $L'(\lambda)v=-v$ for all $v\in X$ and $L(\lambda)\in B(X,Y)$ is
self-adjoint. Hence (ii) follows from Proposition~\ref{eigen.prop}(ii). 
\end{proof}

Our next lemma establishes the hypotheses of 
Theorem~\ref{globif} for the truncated operators $G_n$. 

\begin{lem}\label{propofK} 
Suppose that $f$ satisfies $(\mathrm{f}1)$-$(\mathrm{f}5)$, and 
let $g$ be defined by \eqref{rewrite}. Then, for all $n\in\nat$, the operator $G_n$
defined in \eqref{Gn} has the following properties.
\item[$(\mathrm{i})$] $G_n\in C(X,Y)\cap C^1(X\sm\{0\},Y)$.
\item[$(\mathrm{ii})$] $G_n$ is Fr\'echet differentiable at $u=0$ with $G_n'(0)=0$.
\item[$(\mathrm{iii})$] $G_n$ is compact.  
\end{lem}

\begin{proof} Using the fact that $X\subset L^r(\rn)$ for all $2\les r \les\infty$, the proof
is very similar to that of \cite[Lemmas~4~and~5]{g2} and so we omit it here.
\end{proof}

We will suppose that the hypotheses $(\mathrm{f}1)$ to $(\mathrm{f}6)$ hold for
the rest of this section. 

For $n\in\nat$ and $J:=(\ls,\infty)$, define $F_n:J\x X\to Y$ by 
\begin{equation*}\label{Fn}
F_n(\lambda,v) := L(\lambda)v+G_n(v)
\end{equation*}
and let 
$$
\calZn := \{(\lambda,v)\in J\x X: v\neq0 \ \text{and} \ F_n(\lam,v)=0\},
$$
$$
\calCn := \text{connected component of} \ \calZn\cup\{(\ly,0)\} \  
\text{containing the point} \ (\ly,0).
$$ 
We already know from Lemma~\ref{fredhops} and 
Lemma~\ref{propofK} that we can apply Theorem~\ref{globif} to $F_n$ for any 
$n\in\nat$, with $\lam_0=\ly$. Before doing so, we will first establish some 
preliminary properties of the sets $\calZn$. We start with exponential decay estimates.

\begin{lem}\label{expdecay} Let $(\lam,v) \in J \x X$ satisfy $vL(\lam)v \ges 0$ on $\rn$. 
For any $\ep \in (0,\lam-\ls)$, there exists $C_\ep>0$ such that
$$
|v(x)| \les |v|_\infty \exp[-(\lam-\ls-\ep)^{1/2}(|x|-C_\ep)] \quad \text{for all} \ x \in \rn.
$$ 
\end{lem}

\begin{proof} 
The proof is similar to that of \cite[Lemma~3.4]{sz} but we give it here for completeness.
By the definition of $\ls$ in (f3), for any $\ep \in (0,\lam-\ls)$, there exists $C_\ep>0$ such that
\begin{equation}\label{limsup}
|x|\ges C_\ep \quad \Longrightarrow \quad f_\infty(x) \les \ls + \ep < \lam.
\end{equation}
Now let $\eta:=\lam-\ls-\ep>0$ and define 
$$
z(x):=|v|_\infty \e^{-\sqrt{\eta}(|x|-C_\ep)} - v(x), \quad
\Omega_\ep:=\{ x\in\rn : |x|\ges C_\ep \,, \ z(x)<0 \}.
$$ 
For $x\in\Omega_\ep$, we have $v(x)\ges 0$ and so by our hypothesis and by \eqref{limsup},
$$
0 \les L(\lam)v(x) = \D v(x) - \lam v(x) + f_\infty(x) v(x)
$$
$$
\Longrightarrow \quad  \D v(x) \ges [\lam - f_\infty(x)]v(x) \ges \eta v(x) 
\quad\text{for all} \ x \in \Omega_\ep.
$$
Hence,
\begin{align*}
\D z(x) 	&= 
|v|_\infty \e^{-\sqrt{\eta}(|x|-C_\ep)}\left( \eta - \sqrt{\eta} \, \frac{N-1}{|x|} \right) - \D v(x) \\
			& \les 
|v|_\infty \e^{-\sqrt{\eta}(|x|-C_\ep)}\left( \eta - \sqrt{\eta} \, \frac{N-1}{|x|} \right) - \eta v(x) \\
			& \les \eta z(x) < 0 \quad\text{for all} \ x \in \Omega_\ep.
\end{align*} 
Furthermore, $z(x)=|v|_\infty-v(x)\ges0$ for $|x|=C_\ep$ and $z(x)\to0$ as $|x|\to\infty$. Assuming that
$\Omega_\ep \neq \emptyset$, it follows by the weak maximum principle  
(see e.g \cite[Theorem~8.1]{gt}) that $z\ges0$ on $\Omega_\ep$, a contradiction. Therefore,
$\Omega_\ep = \emptyset$ and 
$$
v(x) \les |v|_\infty \e^{-\sqrt{\eta}(|x|-C_\ep)} \quad\text{for} \ |x| \ges C_\ep.
$$
Applying a similar argument with $-v$ instead of $v$ shows that
$$ 
-v(x) \les |v|_\infty \e^{-\sqrt{\eta}(|x|-C_\ep)} \quad\text{for} \ |x| \ges C_\ep
$$ 
as well, and so $|v(x)| \les |v|_\infty \e^{-\sqrt{\eta}(|x|-C_\ep)}$ for $|x| \ges C_\ep$. Since the inequality
is obvious for $|x| \les C_\ep$, this completes the proof.
\end{proof}

The following {\it a priori} bounds play a central role in the limit procedure.

\begin{lem}\label{bounds}
There is a constant $A>0$ such that, for any $\mu\in J$, there exists $N_\mu\in\nat$ such that
$$
\forall\,\lambda\ges\mu \ \forall\,n\ges N_\mu \quad 
(\lam,v)\in\calZn \quad\Longrightarrow\quad \Ve v\Ve_X \les A.
$$
\end{lem}

\begin{proof}
For $\mu\in J=(\ls,\infty)$, it follows from (f3) that there exists $N_\mu\in\nat$ such that
\begin{equation}\label{bound1}
|x| > N_\mu \quad\Longrightarrow\quad 0 \les f_\infty(x) < \mu.
\end{equation}
Furthermore, by (f2) and (f3), there is an $S>0$ such that
\begin{equation}\label{bound2}
|s| < S \quad\Longrightarrow\quad 0 \les f(x,s) \les f_0(x) + \de < \ls <\mu \quad\text{for all} \ x \in \rn.
\end{equation}
On the other hand, by the Sobolev embedding, there is a constant $C>0$ such that
$|v|_\infty \les C \Ve v\Ve_X$ for all $v\in X$.
Now define $A:=C/S$ and let $(\lambda,v)\in \calZn$ with $\lambda\ges\mu$ and 
$n\ges N_\mu$. Suppose by contradiction that $\Ve v\Ve_X> A$. It follows that
\begin{equation}\label{maj}
\frac{|v(x)|}{\Ve v\Ve_X^2}\les \frac{|v|_\infty}{\Ve v\Ve_X^2}\les\frac{C}{\Ve v\Ve_X}< S
\quad\text{for all} \ x \in \rn.
\end{equation}
The equation $F_n(\lambda,v)=0$ can be written as
\begin{equation*}\label{MP}
\D v(x)v(x)=
\left\{\lambda-[1-\chi_n(x)]f_\infty(x)
-\chi_n(x)f\left(x,\frac{v(x)}{\Ve v\Ve_X^2}\right)\right\}v(x)^2
\quad\text{for all} \ x \in \rn.
\end{equation*}
For $|x|\les n$, $\chi_n(x)=1$ and we have
$$
\D v(x)v(x)=
\left\{\lambda-f\left(x,\frac{v(x)}{\Ve v\Ve_X^2}\right)\right\}v(x)^2\ges
\left\{\mu-f\left(x,\frac{v(x)}{\Ve v\Ve_X^2}\right)\right\}v(x)^2\ges0
$$
by \eqref{bound2}. For $|x|>n\ges N_\mu$, we have
$$
\D v(x)v(x)=\{\lambda-f_\infty(x)\}v(x)^2\ges\{\mu-f_\infty(x)\}v(x)^2\ges0
$$
by \eqref{bound1}. Hence, $\D v(x)v(x)\ges0$ for all $x\in\rn$. 
Using the maximum principle as in the proof of
\cite[Lemma~3.3]{sz}, it follows that $v(x)\equiv0$, a contradiction.
\end{proof}

The following lemma is a first step towards positivity/negativity of solutions.

\begin{lem}\label{sign} For any fixed $n\in\nat$, there is an open neighbourhood
$U$ of $(\lambda_\infty,0)$ in $\real\x X$ such that $v^2>0$ on $[0,\infty)$ for all
$(\lambda,v)\in\calZn\cap U$.
\end{lem}

\begin{proof}
The first part of the proof is similar to that of \cite[Lemma~7]{g2}. Assuming by
contradiction that there is a sequence $\{(\lam_k,v_k)\}\subset\calZn$ such that 
$\lam_k\to\ly$, $\Ve v_k\Ve_X\to0$ and $v_k^2\not>0$ on $\rn$, and defining 
$u_k:=v_k/\Ve v_k\Ve_X$, it follows that there exists $\overline{u}\in X$, with either 
$\overline{u}>0$ or $\overline{u}<0$, such that $u_k\to\overline{u}$ in $X$.
We suppose that $\overline{u}>0$, the other case being similar.
The analogue of equation (17) of \cite{g2} in the present context is
\begin{equation}\label{sign.eq}
\D u_k(x) = \{\lam_k - f_\infty(x) - \chi_n(x)g(x,v_k(x)/\Vert v_k\Vert_X^2)\} u_k(x).
\end{equation}
Since $\lam_k\to\ly$, there exists $k_0$ such that $\lam_k\ges (\ls+\ly)/2>\ls$ 
for all $k\ges k_0$. Furthermore, there exists $R>n$ such that 
$\lam_k-f_\infty(x)-\chi_n(x) g(x,v_k(x)/\Ve v_k\Ve_X^2)>0$ for $|x|>R$ and 
$k\ges k_0$. Also, since $u_k\to\overline{u}$ uniformly on $\rn$ by the Sobolev embedding,
there exist $\delta>0$ and $k_1\ges k_0$ such that $u_k(x)\ges\delta$ for $|x|\les R$ and
$k\ges k_1$. Since $u_k(x)\to0$ as $|x|\to\infty$,
the weak maximum principle now implies that $u_k\ges0$ on $\Omega:=\{x\in\rn: |x|>R\}$, 
for all $k\ges k_1$. Rewriting \eqref{sign.eq} as
$$
-\D u_k + c_+ u_k = c_- u_k
$$
with $c(x):=\lam_k - f_\infty(x) - \chi_n(x)g(x,v_k(x)/\Vert v_k\Vert_X^2)$ and 
$c_\pm:=\max\{\pm c,0\}$, it follows by the strong maximum principle 
\cite[Theorem~8.19]{gt} that $u_k>0$ on $\rn$,  for all $k\ges k_1$. But this
contradicts $v_k^2\not>0$ on $\rn$, and concludes the proof.
\end{proof}

The following lemma describes how elements of $\calZn$ can approach the line of trivial solutions. It can be proved in the same way as Lemma~8 of \cite{g2}.

\begin{lem}\label{approach} Fix $n\in\nat$ and let 
$\{(\lam_k,v_k)\}\subset\calZn$. Suppose that $v_k^2>0$ on $\rn$, 
$\lam_k\to\lambda\in J$ and $\Ve v_k\Ve_X\to0$. Then $\lam=\ly$.
\end{lem}

We now describe the main properties of the component $\calCn$.

\begin{prop}\label{thm1}
\item[$(\mathrm{i})$] For any $n\in\nat$, if $(\lam,v)\in\calCn\sm\{(\ly,0)\}$ then 
$\lam\les\ly$ and $v^2>0$ on $\rn$.
\item[$(\mathrm{ii})$] There is a constant $A>0$ such that, for any $\mu\in J$, there exists $N_\mu\in\nat$ such that
$$
\forall\,n\ges N_\mu \quad \inf\proj\calCn<\mu
\quad\text{and}\quad \Ve v\Ve_X \les A \quad\text{if}\quad (\lam,v)\in\calCn 
\ \text{with} \ \lam\ges\mu.
$$
\end{prop}

\begin{proof} The first part of the proof follows that of Proposition~9 in \cite{g2}, using 
Lemma~\ref{sign}, Lemma~\ref{approach} and similar maximum principle arguments as above to show that the set 
$$
\mathcal{Q}=\{(\lam,v)\in\calCn: v^2>0 \ \text{on} \ \rn\}\cup\{(\ly,0)\}
$$ 
is both open and closed in $\calCn$, with respect to the topology inherited from $\real \x X$. 
Since $\calCn$ is connected, this implies $\calCn=\mathcal{Q}$. 
Furthermore, using Proposition~\ref{eigen.prop}, it follows by a similar proof to that
of \cite[Proposition~14(iv)]{g2} that solutions of $F_n(\lam,v)=0$ with either $v>0$ or $v<0$ 
satisfy $\lam\les\ly$. This proves (i).

The proof of (ii) is identical to that of Proposition~9(ii) in \cite{g2}, using 
Lemma~\ref{bounds} and Lemma~\ref{approach} above.
\end{proof}

We are now in a position to prove global bifurcation for the inverted problem \eqref{opequ}.
We define $F:J\x X\to Y$ by 
\begin{equation*}\label{F}
F(\lambda,v):=L(\lambda)v+G(v), \quad v \in X,
\end{equation*}
and we let
$$
\calZ := \{(\lambda,v)\in J\x X: v^2>0 \ \text{and} \ F(\lam,v)=0\},
$$
$$
\calC := \text{connected component of} \ \calZ\cup\{(\ly,0)\} \  
\text{containing the point} \ (\ly,0).
$$ 
The following estimate is proved in the same way as Lemma~\ref{bounds}, 
replacing $\chi_n$ by 1.

\begin{lem}\label{bounds2}
Let $A>0$ be the constant given by Lemma~\ref{bounds}. For all $(\lam,v)\in J \x X$ such that $F(\lam,v)=0$ we have $\Ve v\Ve_X\les A$.
\end{lem}

\begin{thm}\label{thm2} Suppose $U\subset\real\x X$ is open, bounded, such that 
$(\ly,0)\in U$ and
$$
\mu:=\inf\{\lam:(\lam,v)\in U\}>\ls.
$$
\item[$(\mathrm{i})$] $\calZ\cap\partial U\neq\emptyset$.
\item[$(\mathrm{ii})$] $\calC$ is bounded in $\real\x X$ with
$\inf\proj\calC=\ls$ and $\sup\proj\calC=\ly$.
\item[$(\mathrm{iii})$] Let $\{(\lam_k,v_k)\}\subset \calC$. 
If $\lam_k\to\lam\in J$ and $\Ve v_k\Ve_X\to0$, then
$\lam=\ly$. Conversely, if $\lam_k\to\ly$, then $\Ve v_k\Ve_X\to0$.
\end{thm}

\begin{proof}
Using the preceding results, the proof essentially follows that of 
Theorem~11 in \cite{g2} and we shall only indicate a few details here. 
First note that, in \cite{g2}, we were dealing with positive solutions only.
However, the limit procedure can be carried out in the same way here, yielding solutions with
either $v\ges0$ or $v\les0$ on $\rn$. 
It then follows by the strong maximum principle (as used in the proof of Lemma~\ref{sign})
that such solutions actually satisfy either $v>0$ or $v<0$ on $\rn$. 

Also, it is required in the proof that solutions of 
$F(\lam,v)=0$ with either $v>0$ or $v<0$ satisfy $\lam<\ly$. 
Using the eigenfunction $\ffy>0$ given by Proposition~\ref{eigen.prop}, it 
follows as in the first part of the proof of \cite[Proposition~14(iv)]{g2} that
$$
\ly\intrn \ffy v \diff x = \lam\intrn v \, \ffy \diff x - \intrn h(x,v) v \, \ffy \diff x, 
$$
where $h(x,v(x)):=f(x,v(x)/\Ve v\Ve_X^2)-f_\infty(x)$. Assuming that $v>0$ on $\rn$ and 
remarking that, by (f4), $h\les0$, we see that $\ly\ges\lam$. Then, by (f2) and (f3),
$$
\liminf_{|x|\to\infty} h(x,v(x)) = \liminf_{|x|\to\infty} f_0(x) - \ls \les -\de < 0.
$$
Since $h(x,v(x))$ is continuous, it follows that $h(x,v(x))<0$ on a set of positive 
measure. Hence, $\intrn h(x,v) v\,\ffy \diff x<0$ and we have $\ly>\lam$. 
The case $v<0$ is similar.
\end{proof}


\section{Proof of Theorem~\ref{thm}}

We shall make the hypotheses (f1)-(f6) throughout this section and use the same notations 
as above. 

We prove Theorem~\ref{thm} using the solution set obtained in 
Theorem~\ref{thm2} and the inversion $v\to v/\Ve v\Ve_X^2$. To preserve
connectedness under inversion, we first need to get rid of the right end-point of 
$\calC$. This is done by the following lemma, which can be proved 
similarly to Corollary~5.3 of \cite{sz}.

\begin{lem}\label{C0}
There exists a connected subset $\calCo$ of $\calC\sm\{(\ly,0)\}$ such that 
$\inf\proj\calCo=\ls$ and $(\ly,0)\in\overline{\calCo}$. In particular, 
$\sup\proj\calCo=\ly$ and $0<\Ve v\Ve_X\les A$ for all $(\lam,v)\in\calCo$.
\end{lem}

\begin{rem}
Since $\calZ \subset \real \x C(\rn)$, it follows that the set $\calCo$ obtained in 
Lemma~\ref{C0} lies in one of the sets $\calZ^\pm:= \{(\lambda,v)\in \calZ: \pm v > 0\}$.
\end{rem}

Using the same anti-symmetrization procedure 
as in the beginning of the proof of Theorem~2.3 in \cite{sz},
we obtain two connected subsets $\calCo^\pm$ of $\calC\sm\{(\ly,0)\}$, consisting of
positive/negative solutions and satisfying the properties of $\calCo$ given by 
Lemma~\ref{C0}. Then, from the preceding results, we see that setting
\begin{equation}\label{solset}
\calS^\pm:=\left\{\left(\lam,\frac{v}{\Ve v\Ve_X^2}\right):(\lam,v)\in\calCo^\pm\right\}
\end{equation}
defines connected sets of positive/negative solutions of \eqref{nls}. 
It now follows from Proposition~\ref{reg} that 
$\calS^\pm \subset \real \x H^1(\rn)$ and that property (iv) of Theorem~\ref{thm} holds. 
Furthermore, properties (i) and (ii) follow from Theorem~\ref{thm2}(ii). 
It only remains to prove (iii). 
We already know from Theorem~\ref{thm2}(iii) that 
$\Ve u_n\Ve_X \to \infty \Leftrightarrow \lam=\ly$ 
and, since $- \D +1: X \to Y$ is an isomorphism, it follows from
\begin{equation}\label{finalequ}
-\D u_n + u_n = [f(x,u_n^2) - \lam_n] u_n + u_n
\end{equation}
that $\Ve u_n\Ve_Y \to \infty \Leftrightarrow \Ve u_n\Ve_X \to \infty$. As for $|u_n|_\infty$, 
by the Sobolev embedding, $|u_n|_\infty \to \infty \Rightarrow \Ve u_n\Ve_X \to \infty$. For the
converse, suppose by contradiction that $\Ve u_n \Ve_X \to \infty$ and 
there is a subsequence $u_j:=u_{n_j}$ such
that $|u_j|_\infty$ is bounded. It follows by Lemma~\ref{expdecay} that 
$\{u_j\}$ is bounded in $Y$ and hence, by \eqref{finalequ}, bounded in $X$, a contradiction. 
This concludes the proof of Theorem~\ref{thm}. $\hfill \Box$

\end{document}